\documentclass[pdflatex,sn-mathphys-num]{sn-jnl}

\usepackage{graphicx}%
\usepackage{multirow}%
\usepackage{amsmath,amssymb,amsfonts}%
\usepackage{amsthm}%
\usepackage{mathrsfs}%
\usepackage[title]{appendix}%
\usepackage{xcolor}%
\usepackage{textcomp}%
\usepackage{manyfoot}%
\usepackage{booktabs}%
\usepackage{algorithm}%
\usepackage{algorithmicx}%
\usepackage{algpseudocode}%
\usepackage{listings}%

\theoremstyle{thmstyleone}%
\newtheorem{theorem}{Theorem}%
\newtheorem{proposition}{Proposition}% 
\theoremstyle{thmstyletwo}%
\newtheorem{remark}{Remark}%
\theoremstyle{thmstylethree}%
\theoremstyle{thmstylethree}
\newtheorem{assumption}{Assumption}
\begin{document}

\title[Muon on the Stiefel Manifold Admits an Exact Closed-Form Update]{Muon on the Stiefel Manifold Admits an Exact Closed-Form Update}

\author*[1]{\fnm{Mikhail} \sur{Solonko}}\email{msolonko@hse.ru}

\author[1]{\fnm{Alexander} \sur{Molozhavenko}}

\author[1]{\fnm{Maxim} \sur{Rakhuba}}

\affil[1]{\orgname{HSE University}, \city{Moscow}, \country{Russia}}

\abstract{We study Muon, a recently proposed matrix-aware optimization method, in the context of the Stiefel manifold. This manifold consists of matrices with orthonormal columns and is ubiquitous in machine learning and scientific computing. Existing extensions of Muon to this manifold rely on heuristic, approximate, or iterative updates with varying computational efficiency. We show that the corresponding Stiefel Muon update admits an exact closed-form solution and use this result to develop Skewon, a practical algorithm for orthogonality-constrained optimization with an efficient implementation. We further establish first-order convergence guarantees for Skewon in the smooth non-convex setting.}


\maketitle

\section{Introduction}\label{sec1}
Recently introduced for training deep neural networks, Muon~\cite{jordan2024muon} is an optimizer that has shown promising empirical results, often outperforming standard methods like Adam~\cite{jordan2024muon, shen2025convergence}. 
For the optimization of a smooth objective function, Muon maintains a momentum matrix $M \in \mathbb{R}^{m \times n}$, with $m \geq n$,
and, at each iteration, implicitly solves the following linear minimization
oracle (LMO):
\begin{equation}
\min
\langle M,B\rangle
\quad
\text{subject to}
\quad
\| B\|_2\leq 1,
\label{eq:muon_problem}
\end{equation}
where $\|\cdot\|_2$ denotes the matrix spectral norm. Unlike element-wise optimizers, restricting the spectral norm allows Muon to explicitly account for the underlying matrix structure of the parameters. For a full-rank matrix
$M$, the solution to this problem (which is then used as a weight update) is the negative polar
factor of $M$, obtained from its singular value decomposition:
\begin{equation*}
    B^* = -UV^\top,
    \quad
    M = U\Sigma V^\top.
\end{equation*}

Importantly, Muon can be naturally generalized to optimization problems where the parameter space has a smooth manifold structure. Prime examples are orthogonality constraints widely used across many deep learning architectures. In particular, to integrate the geometric benefits with the efficiency of Muon, recent studies~\cite{bernstein2025manifolds,cesista2025spectralclipping,kexuefm-11221, yang2026manifold, imuon} have investigated its Riemannian adaptations to the Stiefel manifold, defined as
\begin{equation*}
\mathrm{St}(n,p)
=
\left\{
X \in \mathbb{R}^{n\times p}
\;\middle|\;
X^\top X = I_p
\right\},
\quad n>p.
\label{eq:stiefel-manifold}
\end{equation*}
The resulting LMO, which we refer to as the \emph{Stiefel Muon problem} (SMP),  takes the form
\begin{equation}
\min
\langle M,B\rangle
\quad
\text{subject to}
\quad
\| B\|_2\leq 1,
\quad
B\in T_X\mathrm{St}(n,p),
\label{eq:stiefel_muon_problem}
\end{equation}
where $T_X\mathrm{St}(n,p)$ denotes the tangent space at
$X \in \mathrm{St}(n,p)$:
\begin{equation}
T_X\mathrm{St}(n,p)
=
\left\{
B\in\mathbb{R}^{n\times p}
\;\middle|\;
X^\top B+B^\top X=0
\right\}.
\label{eq:stiefel-tangent-space}
\end{equation}

Despite the apparent simplicity of the SMP, existing approaches either solve the tangent-space LMO approximately through iterative procedures~\cite{bernstein2025manifolds,cesista2025spectralclipping, kexuefm-11221}, which introduce significant computational overhead per step, or avoid the exact tangent-space problem via relaxations~\cite{yang2026manifold, imuon}. In contrast, we show that the SMP~\eqref{eq:stiefel_muon_problem} admits an exact closed-form solution and translate this result into a practical algorithm with a computationally efficient implementation, eliminating the need for expensive iterative approximations or geometric relaxations.

Our main contributions are as follows:
\begin{itemize}
\item We formulate an auxiliary optimization problem over skew-symmetric matrices (Section~\ref{sec:A}), derive an exact analytical solution (Section~\ref{sec:A}) and establish its equivalence to the SMP~\eqref{eq:stiefel_muon_problem} (Section~\ref{sec:B}).
\item We introduce \emph{Skewon}, an efficient algorithm for evaluating this analytical solution (Section~\ref{sec:C}).
\item We establish convergence guarantees for Skewon on non-convex smooth objectives under standard assumptions in Riemannian optimization~\cite{boumal2023introduction} (Section~\ref{sec:D}).
\end{itemize}

\subsection{Related Work}
Orthogonality constraints arise naturally across computational mathematics, data analysis, and machine learning. Representative examples include principal component analysis and eigenvalue or invariant-subspace computations~\cite{absil2008optimization,edelman1998geometry}, low-parameter adaptation~\cite{aliev2026orthofuse}, orthogonality-constrained neural networks~\cite{bansal2018can}, and low-rank adaptations with orthogonal factors~\cite{lion2025polar}.
Optimization on the Stiefel manifold has been studied extensively, with efficient first- and second-order Riemannian algorithms described in~\cite{edelman1998geometry, absil2008optimization, boumal2023introduction}. 

In parallel, geometry-aware optimization methods based on spectral geometry have emerged as a promising alternative to conventional Euclidean methods in unconstrained optimization. The Muon optimizer~\cite{jordan2024muon}, for example, performs steepest descent with respect to the spectral norm rather than the Euclidean norm. Its tensor-aware generalization based on a relaxation of the tensor spectral norm was proposed in~\cite{bogachev2026tensorion}. Related work has also explored extensions of the Muon linear minimization oracle (LMO) to Riemannian settings~\cite{bogachev2025lora,imuon}. In particular, adapting Muon to the Stiefel manifold was posed as an open problem in~\cite{bernstein2025manifolds}, primarily because of the difficulty of solving the associated spectral minimization problem~\eqref{eq:stiefel_muon_problem}. Several approaches have recently been proposed. Cesista~\cite{cesista2025spectralclipping} introduced two heuristic methods based on alternating-projection fixed-point iterations and ternary search over nearby feasible points. A fixed-point scheme derived from first-order optimality conditions was proposed in~\cite{kexuefm-11221}, while Bernstein~\cite{bernstein2025manifolds} developed an approach based on Lagrangian duality. Because no simple closed-form solution to the SMP was previously known, these methods rely on iterative procedures and may therefore incur substantial computational cost. Other works avoid the original SMP by replacing it with modified optimization subproblems~\cite{imuon,yang2026manifold}. More broadly, a major computational bottleneck in both Muon and its Stiefel-manifold variants is the evaluation of the matrix sign function (\texttt{msign}), motivating a growing body of work on accelerating the corresponding iterations~\cite{jordan2024muon,grishina2025accelerating,amsel2025polar}.

\subsection{Notation}
Throughout the paper, \(\mathbb{R}^{n\times p}\) denotes the space of real \(n\times p\) matrices equipped with the Frobenius inner product \(\langle X,Y\rangle=\operatorname{tr}(X^\top Y)\),
and the induced Frobenius norm \(\|X\|_F=\sqrt{\langle X,X\rangle}\).

We denote by \(\mathbb{S}^n\) and \(\mathcal{A}^n\) the spaces of symmetric and skew-symmetric \(n\times n\) matrices, respectively. The set of orthogonal \(n\times n\) matrices is denoted by \(O(n)\). For any square matrix \(A\)
\[
\operatorname{sym}(A)=\frac12(A+A^\top),\qquad
\operatorname{skew}(A)=\frac12(A-A^\top).
\]

For \(A\in\mathbb{R}^{m\times n}\), we denote by
\(
\|A\|_2=\sigma_1(A),
\|A\|_*=\sum_i\sigma_i(A)
\)
its spectral and nuclear norms, respectively, where
\(
\sigma_1(A)\ge\sigma_2(A)\ge\cdots\ge0
\)
are the singular values of \(A\). If \(m\ge n\) and \(A\) has full column rank, we define
\[
\operatorname{msign}(A)
=
A(A^\top A)^{-1/2}
=
UV^\top,
\]
where \(A=U\Sigma V^\top\) is a thin singular value decomposition (thin SVD). 

Throughout the paper, \(\operatorname{grad}f(X)\) denotes the Riemannian gradient of a smooth function \(f\) on the Stiefel manifold with respect to the Frobenius metric. The orthogonal projection onto the tangent space at \(X\in\mathrm{St}(n,p)\) is
\[
\mathcal{P}_X(G)
=
G-X\operatorname{sym}(X^\top G),
\]
and \(\mathcal{R}_X:T_X\mathrm{St}(n,p)\rightarrow\mathrm{St}(n,p)\) denotes a retraction.

\section{Skewon problem}\label{sec:A}
In this section, we derive a tractable optimization problem over skew-symmetric matrices and obtain a closed-form solution. The connection between this new formulation and the SMP~\eqref{eq:stiefel_muon_problem} is established in the subsequent sections.

If $Y \in \mathcal{A}^n$, then according to the definition~\eqref{eq:stiefel-tangent-space}, $YX \in T_X \mathrm{St}(n,p)$. Therefore
\[
\left\{
YX \in \mathbb{R}^{n \times p}
\;\middle|\;
Y \in \mathcal{A}^n
\right\} \subset T_X\mathrm{St}(n,p).
\]
In addition, if $\|Y\|_2 \leq 1$, then
\[
\|YX\|_2 \leq \|Y\|_2\|X\|_2 \leq 1.
\]
So, naturally, we can introduce the following problem:
\begin{equation}
\min_{Y} \langle M, YX \rangle
\quad
\mathrm{s.t.}\;
\|Y\|_{2}\le1, \quad
Y\in \mathcal{A}^n,
\label{eq:skewon_problem_0}
\end{equation}
or equivalently
\begin{equation}
\min_{Y} \langle \mathrm{skew} (MX^\top), Y \rangle
\quad
\mathrm{s.t.}\;
\|Y\|_{2}\le1, \quad
Y\in \mathcal{A}^n,
\label{eq:skewon_problem_1}
\end{equation}
when removing the skew-symmetry constraint, we get \emph{Skewon problem}:
\begin{equation}
\min_{Y} \langle \, \operatorname{skew}\left(MX^\top\right), Y\rangle
\quad
\mathrm{s.t.}\;
\|Y\|_{2}\le1.
\label{eq:skewon_problem_2}
\end{equation}
In both problems~\eqref{eq:skewon_problem_1} and~\eqref{eq:skewon_problem_2}, the objective depends only on the skew-symmetric part of \(Y\), because the subspaces of symmetric and skew-symmetric matrices are orthogonal with respect to the Frobenius inner product. Consequently, if \(Y\) is an optimal solution of problem~\eqref{eq:skewon_problem_1}, then \(Y+S\) is optimal for any symmetric matrix \(S\) for problem~\eqref{eq:skewon_problem_2} such that
\[
\|Y+S\|_2 \leq 1.
\]
Moreover, skew-symmetry does not increase the spectral norm. Thus, if we have found the non-skew-symmetric solution to the Skewon problem~\eqref{eq:skewon_problem_2}, then its skew-symmetric part will also be a solution to both problems with the same values of the functionals.

The following proposition shows that the Skewon problem~\eqref{eq:skewon_problem_2} admits an analytical solution, analogous to that of the original Muon LMO.

\begin{proposition}\label{prop:1}
Let $M \in \mathbb{R}^{n \times p}$ be arbitrary, $X \in \mathrm{St}(n,p)$ and denote $N = \operatorname{skew}(MX^\top)$. Let also $N = U \Sigma V^\top$  be an SVD of \(N\) and \(r = \mathrm{rank} \: N\):
\begin{equation*}
    N = \begin{bmatrix} U_1 & U_2 \end{bmatrix} 
    \begin{bmatrix} \Sigma_r & \mathbf{0} \\ \mathbf{0} & \mathbf{0} \end{bmatrix} 
    \begin{bmatrix} V_1^\top \\ V_2^\top \end{bmatrix}.
\end{equation*}
The optimal objective value for the Skewon problem~\eqref{eq:skewon_problem_2} is $(-\|N\|_*)$, and the general optimal solution in terms of the SVD is given by
\begin{equation} \label{eq:Y_opt_svd}
    Y_{\mathrm{opt}} = -U \begin{bmatrix} 
        I_r & \mathbf{0} \\ \mathbf{0} & W 
    \end{bmatrix} V^\top = -U_1 V_1^\top - U_2 W V_2^\top,
\end{equation}
where $\|W\|_2 \le 1$. 

In particular, if the matrix $N$ is nonsingular, then the optimal
solution to the Skewon problem~\eqref{eq:skewon_problem_2} is given by
\begin{equation*}
    Y_{\mathrm{opt}}
    =
    -\operatorname{msign}\left(
        N
    \right) = - UV^\top.
\end{equation*}
\end{proposition}

\begin{proof} 
Let $Y$ be any feasible solution. Define $\hat{Y} = -U^\top Y V$. Since the spectral norm is unitarily invariant, we have $\|\hat{Y}\|_2 = \|Y\|_2 \le 1$. We can express the objective function in terms of $\hat{Y}$:
\begin{equation*}
    \langle N, Y \rangle 
    = \operatorname{tr}(N^\top Y) 
    = \operatorname{tr}(V \Sigma U^\top Y) 
    = \operatorname{tr}(\Sigma U^\top Y V) 
    = -\operatorname{tr}(\Sigma \hat{Y}) 
    = -\sum_{i=1}^r \sigma_i(N) \hat{Y}_{ii}.
\end{equation*}
Since $\|\hat{Y}\|_2 \le 1$, we have the lower bound:
\begin{equation*}
    \langle N, Y \rangle = -\sum_{i=1}^r \sigma_i(N) \hat{Y}_{ii} \ge -\sum_{i=1}^r \sigma_i(N) \cdot 1 = -\|N\|_*.
\end{equation*}
This establishes that the optimal objective value is indeed $-\|N\|_*$. 

For a matrix $Y$ to be optimal, it must achieve this lower bound with equality. This implies that $\hat{Y}_{ii} = 1$ for all $i = 1, \dots, r$. If a matrix has a spectral norm bounded by $1$ and a diagonal element equal to $1$, all other entries in the corresponding row and column must be exactly zero. Therefore, $\hat{Y}$ must have the block-diagonal structure:
\begin{equation*}
    \hat{Y} = \begin{bmatrix} I_r & \mathbf{0} \\ \mathbf{0} & W \end{bmatrix},
\end{equation*}
for some arbitrary block $W$ satisfying $\|W\|_2 \le 1$. Substituting $Y = -U \hat{Y} V^\top$ yields the general optimal solution:
\begin{equation*}
    Y_{\mathrm{opt}} = -U \begin{bmatrix} I_r & \mathbf{0} \\ \mathbf{0} & W \end{bmatrix} V^\top = -U_1 V_1^\top - U_2 W V_2^\top,
\end{equation*}
which proves that no other optimal solutions exist.

Furthermore, if \(N\) is nonsingular, then \(r=n\), and the blocks
\(U_2\), \(V_2\), and \(W\) are absent. The general solution~\eqref{eq:Y_opt_svd} then reduces uniquely to:
\begin{equation*}
    Y_\mathrm{opt} = -U I_n V^\top = -U V^\top = -\operatorname{msign}(N),
\end{equation*}
which concludes the proof.
\end{proof}

\begin{remark}\label{corollary:boundary_optimum_skewon}
As we can see in equation~\eqref{eq:Y_opt_svd}, if the matrix \(N\) is nonzero (which is equivalent to \(\mathcal{P}_XM \neq 0\)), then the solution of the Skewon problem~\eqref{eq:skewon_problem_2} is achieved on the boundary of the spectral-norm unit ball.
\end{remark}

For subsequent proofs, we will need one more auxiliary proposition.

\begin{proposition}\label{lemma:boundary_optimum_stifel_muon}
Every optimal solution of the SMP~\eqref{eq:stiefel_muon_problem} with nonzero $\mathcal{P}_XM$ is attained on the boundary of the
spectral-norm unit ball, i.e. $\left\|B_{\mathrm{opt}}\right\|_2 = 1$.
\end{proposition}
\begin{proof}
See Section~\ref{lemma:boundary_optimum_stifel_muon_proof}.
\end{proof}

\section{Equivalence of Stiefel Muon and Skewon problems}\label{sec:B}
In this section, we investigate the relationship between \(\mathcal{A}^n\) and \(T_X\mathrm{St}(n,p)\). As a result, we show that the Skewon problem~\eqref{eq:skewon_problem_2} is not a relaxation of the SMP~\eqref{eq:stiefel_muon_problem}, but rather an \emph{equivalent reformulation}.

\begin{theorem}[Direct sum decomposition of $\mathcal{A}^n$]\label{thm:stiefel_decomposition}
Let $X \in \mathrm{St}(n,p)$. The subspace of skew-symmetric matrices $\mathcal{A}^n$ decomposes into the direct sum of two orthogonal subspaces:
\[
\mathcal{A}^n = \mathcal{L}_X \oplus \mathcal{K}_X,
\]
where
\begin{align}
\mathcal{L}_X &\mathrel{:=} \left\{ BX^\top - XB^\top - X(X^\top B)X^\top \;\middle|\; B \in T_X\mathrm{St}(n,p) \right\}\label{eq:dsd_rotation}, \\
\mathcal{K}_X &\mathrel{:=} \left\{ (I_n - XX^\top)C(I_n - XX^\top) \;\middle|\; C \in \mathcal{A}^n \right\}\label{eq:dsd_stabilizer}.
\end{align}
In particular, any matrix $Y \in \mathcal{A}^n$ uniquely decomposes as $Y = \tilde{Y} + Y_{\mathrm{ker}}$ with $\tilde{Y} \in \mathcal{L}_X$ and $Y_{\mathrm{ker}} \in \mathcal{K}_X$, satisfying:
\begin{enumerate}
    \item $Y_{\mathrm{ker}} X = \mathbf{0}$,
    \item $\tilde{Y} X = YX = B \in T_X\mathrm{St}(n,p)$.
\end{enumerate}
\end{theorem}

\begin{proof}
Let $Q = \begin{bmatrix} X & X_\perp \end{bmatrix} \in O(n)$ be an orthogonal matrix obtained by completing $X$. Any skew-symmetric matrix $Y \in \mathcal{A}^n$ can be represented in the basis $Q$ as
\[
Y = Q \begin{bmatrix} 
    A & -S^\top \\ 
    S & D 
\end{bmatrix} Q^\top,
\]
where $A \in \mathcal{A}^p$, $D \in \mathcal{A}^{n-p}$, and $S \in \mathbb{R}^{(n-p) \times p}$. We can naturally split $Y$ into two components $Y = \tilde{Y} + Y_{\mathrm{ker}}$, defined by:
\begin{equation}\label{eq:dsd_stabilizer_and_rotation_explicit_repr}
\tilde{Y} \mathrel{:=} Q \begin{bmatrix} 
    A & -S^\top \\ 
    S & \mathbf{0} 
\end{bmatrix} Q^\top, \quad 
Y_{\mathrm{ker}} \mathrel{:=} Q \begin{bmatrix} 
    \mathbf{0} & \mathbf{0} \\ 
    \mathbf{0} & D 
\end{bmatrix} Q^\top.
\end{equation}

Direct calculation shows that $Y_{\mathrm{ker}} = (I_n - XX^\top) C (I_n - XX^\top) \in \mathcal{K}_X$ for $C = Y$, and $Y_{\mathrm{ker}} X = \mathbf{0}$ because $X = Q \begin{bmatrix} I_p & \mathbf{0} \end{bmatrix}^\top$.

Multiplying $\tilde{Y}$ by $X$ yields:
\[
\tilde{Y} X = Y X = Q \begin{bmatrix} 
    A & -S^\top \\ 
    S & \mathbf{0} 
\end{bmatrix} \begin{bmatrix} I_p \\ \mathbf{0} \end{bmatrix} = Q \begin{bmatrix} A \\ S \end{bmatrix} = X A + X_\perp S \mathrel{=:} H.
\]
Since $A \in \mathcal{A}^p$, we have $X^\top H = A$, which is skew-symmetric. Thus, $ H \in T_X\mathrm{St}(n,p)$. Substituting  $H$ back into the definition of $\mathcal{L}_X$~\eqref{eq:dsd_rotation} confirms that $\tilde{Y} \in \mathcal{L}_X$.

Conversely, we must verify that any element in $\mathcal{L}_X$ has $D = \mathbf{0}$, and any element in $\mathcal{K}_X$ has $A = \mathbf{0}, S = \mathbf{0}$.

For $Y_\mathrm{ker} \in \mathcal{K}_X$, since
\[
I_n - XX^\top = Q \begin{bmatrix} \mathbf{0} & \mathbf{0} \\ \mathbf{0} & I_{n-p} \end{bmatrix} Q^\top,
\]
we have
\[
Y_\mathrm{ker} = Q \begin{bmatrix} \mathbf{0} & \mathbf{0} \\ \mathbf{0} & I_{n-p} \end{bmatrix} Q^\top Q \begin{bmatrix} 
    A & -S^\top \\ 
    S & D 
\end{bmatrix} Q^\top Q \begin{bmatrix} \mathbf{0} & \mathbf{0} \\ \mathbf{0} & I_{n-p} \end{bmatrix} Q^\top= Q \begin{bmatrix} 
    \mathbf{0} & \mathbf{0} \\ 
    \mathbf{0} & D 
\end{bmatrix} Q^\top.
\]

For $\tilde Y \in \mathcal{L}_X$ we have
\[
Q^\top \tilde Y Q = \begin{bmatrix} X^\top B & \mathbf{0} \\ X_\perp^\top B & \mathbf{0} \end{bmatrix} - \begin{bmatrix} B^\top X & B^\top X_\perp \\ \mathbf{0} & \mathbf{0} \end{bmatrix} - \begin{bmatrix} X^\top B & \mathbf{0} \\ \mathbf{0} & \mathbf{0} \end{bmatrix} = \begin{bmatrix} X^\top B & -B^\top X_\perp  \\ X_\perp^\top B & \mathbf{0} \end{bmatrix},
\]
where $X^\top B \in \mathcal{A}^p$. So,
\[
\tilde Y = Q \begin{bmatrix} X^\top B & -B^\top X_\perp \\ X_\perp^\top B & \mathbf{0} \end{bmatrix} Q^\top.
\]

If $Y \in \mathcal{L}_X \cap \mathcal{K}_X$, then $Y$ must simultaneously have $D = \mathbf{0}$ (as $Y \in \mathcal{L}_X$) and $A = \mathbf{0}, S = \mathbf{0}$ (as $Y \in \mathcal{K}_X$). Thus, all block entries vanish, implying $Y = \mathbf{0}$. Since the block components $A \in \mathcal{A}^p$, $S \in \mathbb{R}^{(n-p) \times p}$, and $D \in \mathcal{A}^{n-p}$ are independent, their degrees of freedom yield:

\[
    \dim \mathcal{L}_X = \frac{p(p-1)}{2} + p(n-p), \quad \dim \mathcal{K}_X = \frac{(n-p)(n-p-1)}{2}.
    \]
    Summing these dimensions gives:
    \[
    \dim \mathcal{L}_X + \dim \mathcal{K}_X = \left(\frac{p(p-1)}{2} + p(n-p)\right) + \frac{(n-p)(n-p-1)}{2} = \frac{n(n-1)}{2} = \dim \mathcal{A}^n.
    \]

Since the intersection is trivial and the dimensions sum up to $\dim \mathcal{A}^n$, the decomposition is a direct sum and the components $(\tilde{Y}, Y_{\mathrm{ker}})$ are unique. The explicit representation~\eqref{eq:dsd_stabilizer_and_rotation_explicit_repr} of \(\mathcal{L}_X\) and \(\mathcal{K}_X\) ensures the orthogonality of the subspaces~\eqref{eq:dsd_rotation} and \eqref{eq:dsd_stabilizer} with respect to Frobenius inner product.
\end{proof}

\begin{remark}
Geometrically, viewing $\mathcal{A}^n \equiv \mathfrak{so}(n)$ as the Lie algebra of the rotation group $O(n)$: $\mathcal{L}_X \cong T_X\mathrm{St}(n,p)$ represents infinitesimal rotations that actively tilt and shift the frame $X$, whereas $\mathcal{K}_X \cong \mathfrak{so}(n-p)$ forms the stabilizer Lie algebra of rotations acting exclusively within the complementary subspace $\mathrm{span}(X_\perp)$ that leave $X$ invariant. Although $Y_{\mathrm{ker}}$ does not change the resulting tangent vector $YX = B$, it shifts the matrix spectrum and affects the spectral norm $\|Y\|_2$. As we show next, these properties allow us to explicitly map solutions between Skewon formulation~\eqref{eq:skewon_problem_2} and the SMP~\eqref{eq:stiefel_muon_problem}.
\end{remark}

\begin{proposition}[Davis--Kahan--Weinberger {\cite{davis1982norm}}]\label{lemma-dkw}
Let
\[
T=
\begin{bmatrix}
A&C\\
B&D
\end{bmatrix}
\]
be a block operator. Then for fixed blocks \(A, B\) and \(C\) there exists the block \(D\) such that
\[
\left\|
\begin{bmatrix}
A&C\\
B&D
\end{bmatrix}
\right\|_2
=
\max\left\{
\left\|
\begin{bmatrix}
A&C
\end{bmatrix}
\right\|_2,
\left\|
\begin{bmatrix}
A\\
B
\end{bmatrix}
\right\|_2
\right\}.
\]
\end{proposition}
\begin{theorem}[Equivalence of Optima]\label{theorem:equiv_of_opt}
Let $B_*$ be an optimal solution to the SMP~\eqref{eq:stiefel_muon_problem} with nonzero $\mathcal{P}_XM$. There exists an optimal solution $Y_* \in \mathcal{A}^n$ to Skewon problem~\eqref{eq:skewon_problem_2} such that $Y_* X = B_*$, $\|Y_*\|_2 = \|B_*\|_2 = 1$, and both problems achieve identical minimum objective values.
\end{theorem}

\begin{proof}
Using the direct sum decomposition of $\mathcal{A}^n$ (Theorem~\ref{thm:stiefel_decomposition}), we can define a family of matrices $Y(C) \in \mathcal{A}^n$ that map to $B_*$ via $Y(C)X = B_*$:
\begin{equation*}
Y(C) = B_*X^\top - XB_*^\top - X(X^\top B_*)X^\top + (I_n - XX^\top)C(I_n - XX^\top) = \tilde{Y}_* + Y_{\mathrm{ker}}(C),
\end{equation*}
where $C \in \mathcal{A}^n$.

First, we verify that the kernel component $Y_{\mathrm{ker}}(C)$ does not affect the value of the objective functional. Since $Y_{\mathrm{ker}}(C)X = \mathbf{0}$, the objective remains invariant:
\begin{align*}
&\frac{1}{2}\left\langle MX^\top - XM^\top , Y_{\mathrm{ker}} \right\rangle = \langle M, Y_{\mathrm{ker}} X \rangle = 0.
\end{align*}

However, the kernel component does dictate the spectral norm $\|Y(C)\|_2$. Because the optimal solutions for both problems are inherently attained on the boundary of the spectral-norm unit ball, i.e., $\|B_*\|_2 = 1$ (see Proposition~\ref{lemma:boundary_optimum_stifel_muon} and Remark~\ref{corollary:boundary_optimum_skewon}), it suffices to prove that there exists a suitable matrix $C \in \mathcal{A}^n$ such that $\|Y(C)\|_2 = \|B_*\|_2$.

Consider the matrix $Q^\top Y(C) Q$ after applying a similarity transformation with $Q = \begin{bmatrix} X & X_\perp \end{bmatrix}$:
\begin{equation*}
Q^\top Y(C) Q = \begin{bmatrix} X^\top B_* & -B_*^\top X_\perp \\ X_\perp^\top B_* & X_\perp^\top C X_\perp \end{bmatrix}
\end{equation*}
The block $\hat{C} = X_\perp^\top C X_\perp \in \mathcal{A}^{n-p}$ represents the projection onto the complementary subspace. Let us construct an intermediate matrix $H_0$ by replacing $\hat{C}$ with an arbitrary unconstrained block $D \in \mathbb{R}^{(n-p)\times (n-p)}$:
\begin{equation*}
H_0 = \begin{bmatrix} X^\top B_* & -B_*^\top X_\perp \\ X_\perp^\top B_* & D \end{bmatrix}
\end{equation*}
By Proposition~\ref{lemma-dkw}, there exists a block $D$ that attains the minimum possible completion norm:
\begin{equation*}
\|H_0\|_2 = \max\left\{ \left\| \begin{bmatrix} X^\top B_* & -B_*^\top X_\perp \end{bmatrix} \right\|_2, \left\| \begin{bmatrix} X^\top B_* \\ X_\perp^\top B_* \end{bmatrix} \right\|_2 \right\}
\end{equation*}

We evaluate these block norms using the unitary invariance of the spectral norm. For the horizontal block:
\begin{align*}
\left\|\begin{bmatrix} X^\top B_* & -B_*^\top X_\perp \end{bmatrix} \right\|_2 &= \left\|\begin{bmatrix} X^\top B_* & -B_*^\top X_\perp \end{bmatrix} Q^\top\right\|_2 = \\ 
&= \left\| X^\top B_* X^\top - B_*^\top X_\perp X_\perp^\top \right\|_2 = \\ 
&= \left\| - B_*^\top X X^\top - B_*^\top(I_n - X X^\top) \right\|_2 = \\ 
&= \left\|-B_*^\top\right\|_2 = \|B_*\|_2.
\end{align*}
Similarly, for the vertical block:
\begin{equation*}
\left\| \begin{bmatrix} X^\top B_* \\ X_\perp^\top B_* \end{bmatrix} \right\|_2 = \|B_*\|_2.
\end{equation*}
Thus, there exists a block $D$ such that $\|H_0\|_2 = \|B_*\|_2$. 

Because this specific block $D$ is not necessarily skew-symmetric, we define $H = \operatorname{skew}(H_0)$. Under this operation, the sub-block $D$ becomes $\operatorname{skew}(D)$, while the other blocks remain unchanged. On the one hand, by the triangle inequality:
\begin{equation*}
\|H\|_2 = \frac{1}{2} \left\|H_0 - H_0^\top\right\|_2 \le \|H_0\|_2 = \|B_*\|_2
\end{equation*}
On the other hand, the operator norm of a submatrix can never exceed the norm of the full matrix:
\begin{equation*}
\|H\|_2 \ge \left\|\begin{bmatrix} X^\top B_* & -B_*^\top X_\perp \end{bmatrix} \right\|_2 = \|B_*\|_2
\end{equation*}
Therefore, we conclude that $\|H\|_2 = \|B_*\|_2$. 

Finally, by choosing $C$ such that $\hat{C} = X_\perp^\top C X_\perp = \operatorname{skew}(D)$ (which is valid since projection onto a subspace is a surjective mapping), we construct a matrix $Y_* = Y(C) \in \mathcal{A}^n$ satisfying:
\begin{equation*}
    \|Y_*\|_2 = \left\|Q^\top Y_* Q\right\|_2 = \|H\|_2 = \|B_*\|_2.
\end{equation*}

By Proposition~\ref{lemma:boundary_optimum_stifel_muon}, the optimal Stiefel tangent vector attains the boundary of the unit ball, meaning $\|B_*\|_2 = 1$. Consequently, $\|Y_*\|_2 = 1$, which makes $Y_*$ a feasible solution for the Skewon problem. Furthermore, since $Y_* X = B_*$, the objective value achieved by $Y_*$ in the Skewon problem exactly matches the optimal value of the SMP:
\begin{equation*}
    \langle M, Y_* X \rangle = \langle M, B_* \rangle.
\end{equation*}

To establish that $Y_*$ is indeed the global optimum for the Skewon problem and that the minimum values coincide, let $Y_{\mathrm{opt}} \in \mathcal{A}^n$ be any optimal solution to the Skewon problem~\eqref{eq:skewon_problem_2} (such skew-symmetric solution always exists since the symmetric part does not affect the value of the functional and skew-symmetry does not increase the spectral norm). Consider the matrix $B_{\mathrm{opt}} = Y_{\mathrm{opt}} X \in T_X\mathrm{St}(n,p)$.

Since $Y_{\mathrm{opt}}\in\mathcal{A}^n$ is feasible for the Skewon problem~\eqref{eq:skewon_problem_2}, we have $\Vert{}Y_{\mathrm{opt}}\Vert{}_2 \le 1$. Consequently, $B_{\mathrm{opt}} \in T_X \mathrm{St}(n,p)$ and its spectral norm is bounded by $\Vert{}B_{\mathrm{opt}}\Vert{}_2 \le \Vert{}Y_{\mathrm{opt}}\Vert{}_2 \Vert{}X\Vert{}_2 \le 1$. This makes $B_{\mathrm{opt}}$ a valid feasible solution for the SMP. Therefore, the optimal value of the SMP forms a lower bound for the Skewon problem:$$\langle M, B_* \rangle \le \langle M, B_{\mathrm{opt}} \rangle = \langle M, Y_{\mathrm{opt}} X \rangle$$However, we have already constructed a feasible Skewon solution $Y_*$ that achieves exactly $\langle M, Y_* X \rangle = \langle M, B_* \rangle$. This implies:$$\langle M, Y_{\mathrm{opt}} X \rangle \le \langle M, Y_* X \rangle = \langle M, B_* \rangle$$Combining these inequalities yields $\langle M, Y_{\mathrm{opt}} X \rangle = \langle M, B_* \rangle$. Thus, $Y_*$ is a global optimum for the Skewon problem, and both formulations achieve identical minimum objective values, completing the proof of equivalence.
\end{proof}
\begin{remark}
The case $\mathcal{P}_XM = \mathbf{0}$ is trivial: one may choose $Y_*=\mathbf{0}$, and both problems attain the same optimal value 0.
\end{remark}

\begin{remark}
The proof of Theorem~\ref{theorem:equiv_of_opt} establishes the stronger result
\[
\left\{YX \colon Y^\top = -Y,\ \|Y\|_2 \leq 1\right\}
=
\left\{B \in T_X \mathrm{St}(n,p) \colon \|B\|_2 \leq 1\right\}.
\]
\end{remark}

\section{Skewon algorithm}\label{sec:C}

\begin{algorithm}[htbp]
\caption{Skewon (case $p \sim n $)}\label{alg:1}
\begin{algorithmic}[1]
\Require Initial point $X_0\in \mathrm{St}(n,p)$, number of iterations $T$, step size $\{\eta_t\}_{t=0}^{T-1}$
\Ensure $X_T$
\For{$t=0,1,\ldots,T-1$}
    \State Compute direction $M_t$
    \State $N_t = \mathrm{skew}\left(M_tX_t^\top\right)$ \Comment{$\mathcal{O}(n^2p)$} 
    \State $Y_t= -\mathrm{skew}\left(\mathrm{NS}\left(N_t \right)\right)$ \Comment{$\mathcal{O}(n^3)$}
    \State $B_t = Y_t X_t$ \Comment{$\mathcal{O}(n^2p)$}
    \State $X_{t+1} = \mathcal{R}_{X_t}\big(\eta_t B_t\big)$ \Comment{$\mathcal{O}(np^2)$}
\EndFor
\end{algorithmic}
\end{algorithm}

\begin{algorithm}[htbp]
\caption{Skewon (case $p \ll n$)}\label{alg:2}
\begin{algorithmic}[1]
\Require Initial point $X_0\in \mathrm{St}(n,p)$, \(2p \leq n\), number of iterations $T$, step size $\{\eta_t\}_{t=0}^{T-1}$
\Ensure $X_T$
\For{$t=0,1,\ldots,T-1$}
    \State Compute direction $M_t$
    \State $Q_t, R_t = \mathrm{QR}\left(\begin{bmatrix}
        -X_t & M_t
    \end{bmatrix}\right)$ \Comment{$\mathcal{O}(np^2)$}
    \State $\tilde Y_t= -\mathrm{skew}\left(\mathrm{NS}\left(\frac{1}{2}R_tJR_t^\top \right)\right)$ \Comment{$\mathcal{O}(p^3)$}
    \State $B_t = Q_t\left(\tilde Y_t(Q_t^\top X_t)\right)$ \Comment{$\mathcal{O}(np^2)$}
    \State $X_{t+1} = \mathcal{R}_{X_t}\big(\eta_t B_t\big)$ \Comment{$\mathcal{O}(np^2)$}
\EndFor
\end{algorithmic}
\end{algorithm}

In this section, we develop an efficient numerical procedure for solving the Skewon problem~\eqref{eq:skewon_problem_2}. Algorithm~\ref{alg:1} provides the simplest implementation of the proposed optimization procedure. In Line~\(2\), it computes an optimization direction \(M_t\), such as a gradient or momentum direction. In Lines~\(3\)--\(4\), it evaluates the solution to the Skewon problem~\eqref{eq:skewon_problem_2} $N_t$ from Proposition~\ref{prop:1} using the Newton--Schulz iteration. The additional skew-symmetrization in Line~\(4\) is required to mitigate round-off errors introduced by the Newton--Schulz step. In Line~\(5\), the algorithm maps the solution of the Skewon problem~\eqref{eq:skewon_problem_2} to a solution of the SMP~\eqref{eq:stiefel_muon_problem}. Finally, in Line~\(6\), it performs an update using a retraction. 

The Newton--Schulz iteration in Line~\(4\) may be replaced with more advanced methods, such as Polar Express~\cite{amsel2025polar} or CANS~\cite{grishina2025accelerating}. For greater numerical accuracy, the polar factor may instead be computed exactly using a thin SVD with singular value thresholding, which removes spurious directions introduced by skew-symmetrization (see Proposition~\ref{prop:skewSVD} in Section~\ref{sec:proofs}). In practice this approach is generally less GPU-friendly than the Newton--Schulz iteration and may, in some cases, degrade performance~\cite{gonon2026insights}. 
Different retraction maps may also be used in Line~\(6\), including the \(Q\)-factor obtained from a QR decomposition and the polar factor computed either iteratively or via the SVD.

When \(p \ll n\), Algorithm~\ref{alg:1} can be improved by evaluating the solution to the Skewon problem~\eqref{eq:skewon_problem_2} more efficiently in Lines~\(3\)--\(4\). In particular, the matrix \(\operatorname{skew}(MX^\top)\) can be expressed as follows:
\begin{equation*}
    \operatorname{skew}(MX^\top) = \frac{1}{2}\left(-XM^\top + MX^\top \right) 
    = \frac{1}{2} \begin{bmatrix} -X & M \end{bmatrix} 
    \underbrace{\begin{bmatrix} \mathbf{0} & I_p \\ -I_p & \mathbf{0} \end{bmatrix}}_{J} 
    \begin{bmatrix} -X & M \end{bmatrix}^\top.
\end{equation*}
If $2p\leq n$ (which is the case if \(p \ll n\)) by applying a QR decomposition 
\[
    \begin{bmatrix} -X & M \end{bmatrix} = Q R, \quad Q \in \mathbb{R}^{n \times 2p}, \quad R \in \mathbb{R}^{2p \times 2p}.
\]
we obtain:
\begin{equation*}
    \operatorname{skew}(MX^\top) = Q \left( \frac{1}{2} R J R^\top \right) Q^\top.\label{eq:algend}
\end{equation*}

An important structural advantage of this formulation is the smaller inner matrix $ R J R^\top\in \mathcal{A}^{2p}$. The resulting procedure is summarized in Algorithm~\ref{alg:2}.

Algorithm~\ref{alg:2} computes, in Line~3, a QR decomposition of the block matrix
\(\begin{bmatrix}
    -X_t & M_t
\end{bmatrix}\).
This keeps the computational complexity at \(\mathcal{O}(np^2)\). In Line~4, the Skewon problem~\eqref{eq:skewon_problem_2} is then solved using the smaller \(2p \times 2p\) matrix \(\frac{1}{2}R_t J R_t^\top\), which can be handled more efficiently and requires only \(\mathcal{O}(p^3)\) operations. 

Additionally, we can obtain a potential speedup in Algorithm~\ref{alg:2} if we use the structure of the block matrix $\begin{bmatrix}
        -X_t & M_t
\end{bmatrix}$. The matrix $-X_t$ already has orthonormal columns. We can use block classical Gram--Schmidt with reorthogonalization (BCGS2)~\cite{barlow2013reorthogonalized} to orthonormalize \(M_t\) against $X_t$, which simplifies to matrix multiplications and QR decomposition of an \(n \times p\) matrix. 

Combining all steps, the per-iteration complexity when $p \ll n$ is $\mathcal{O}(np^2)$ and $\mathcal{O}(n^3)$ otherwise. Thus, from a practical deep learning perspective, Skewon has the same asymptotic computational complexity as Muon while solving the optimization subproblem directly in the tangent space of the Stiefel manifold. 

\section{Convergence guarantees}\label{sec:D}
In this section, we establish first-order convergence guarantees for both Algorithms~\ref{alg:1} and~\ref{alg:2} for smooth non-convex objectives when the matrix $M_t$ is Euclidean gradient $\nabla f(X_t)$ and the polar factor computed via thin SVD. Note that since $\mathrm{St}(n,p)$ is compact for any smooth objective function $f$ there exists $f_{\mathrm{opt}} > -\infty$ such that $f(X) \geq f_{\mathrm{opt}}$ for all $X \in \mathrm{St}(n,p)$. To proceed with the analysis, we first introduce the following Assumption~\cite{boumal2023introduction}:
\begin{assumption}\label{ass:1}
     For any $X \in \mathrm{St}(n,p)$ and any tangent vector $B \in T_X \mathrm{St}(n,p)$, the following inequality holds
\begin{align*}
f(\mathcal{R}_X(B)) \leq f(X) + \langle \mathrm{grad} f(X) , B\rangle + \frac{L}{2} \|B\|^2_F.
\end{align*}
\end{assumption}

\begin{theorem}\label{theorem:convergence_garancy}
Under Assumption~\ref{ass:1}, for any realization $(X_t)_{t\in \mathbb{N}}$ of Algorithm~\ref{alg:1} and Algorithm~\ref{alg:2}, in which the polar factor is evaluated using a thin SVD, let $\Delta \geq f(X_0) - \inf_{X \in \mathrm{St}(n,p)} f(X)$, and choose the optimization step size $\eta = \sqrt{{2\Delta}/{pLT}}$. Then,
\[
\min_{t\in\{0,\ldots,T-1\}} \left\|\mathrm{grad}f(X_t)\right\|_F \leq \sqrt{\frac{8p\Delta L}{T}}.
\]
\end{theorem}

\begin{proof}
See Section~\ref{theorem:convergence_garancy_proof}.
\end{proof}

\section{Numerical experiments}

We compare Algorithm~\ref{alg:2} for solving the SMP~\eqref{eq:stiefel_muon_problem} with the methods proposed by Bernstein~\cite{bernstein2025manifolds}, Cesista~\cite{cesista2025spectralclipping}, and Su~\cite{kexuefm-11221}, using the SVD-based variant of Su's algorithm. We also consider Riemannion~\cite{bogachev2025lora} and iMuon~\cite{imuon} algorithms. We compare the runtime and relative error for different values of $n$ and $p$. For all methods that require an inner iteration, we use five inner-loop steps. 

In our Algorithm~\ref{alg:2}, the QR decomposition is computed using block BCGS2. We also use exact matrix-sign-function solvers throughout Algorithm~\ref{alg:2} and do not consider iterative approximations such as Newton--Schulz iterations.

All experiments were conducted on an Intel Core i5-10400F 2.90GHz CPU. Points on the Stiefel manifold were generated as the $Q$ factors of random matrices with independent standard normal entries, and the gradient matrices were sampled independently from the same distribution. For each pair $(n,p)$, we performed 50 runs and report the median results. The reference solution was obtained by solving the SMP~\eqref{eq:stiefel_muon_problem} with SCS solver~\cite{ocpb:16} (\(10^{-11}\) both absolute and relative tolerance) in \texttt{CVXPY}~\cite{diamond2016cvxpy}. Relative error was measured in the Frobenius norm:
\[
\frac{\|B_{\texttt{CVXPY}} - B_{\mathrm{alg}}\|_F}{\|B_{\texttt{CVXPY}}\|_F}.
\]
Figures~\ref{fig:exp1_1} and~\ref{fig:exp1_2} show the experimental results.

We also study the dependence of the relative error on the running time of the algorithm, changing the number of inner iterations. The result is shown in Figure~\ref{fig:exp1_3}.

\begin{figure}[h]
    \centering
    \includegraphics[width=\textwidth]{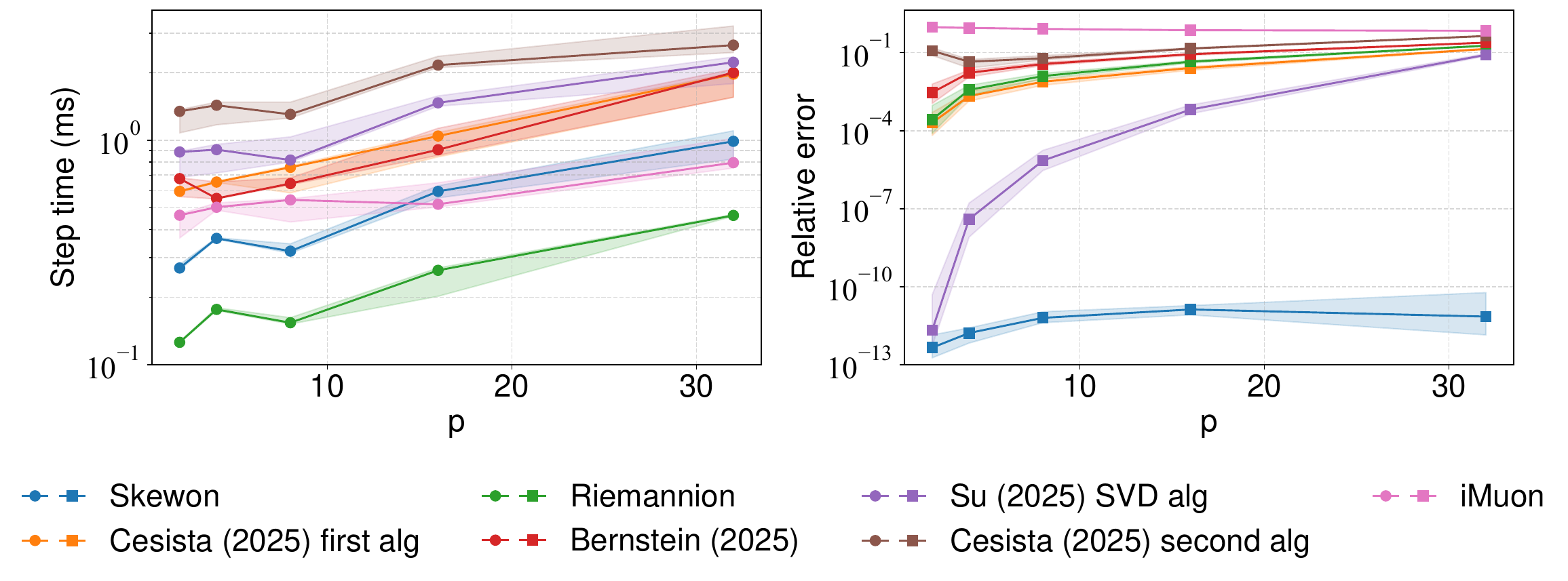}
    \caption{The left graph shows the dependence of the SMP~\eqref{eq:stiefel_muon_problem} solution median time on \(p\). The right graph shows the median relative error of the SMP~\eqref{eq:stiefel_muon_problem} solution at different \(p\). Parameter \(n\) is fixed and equal to \(64\). The range of $0.25$ and $0.75$ quantiles is shaded.}
    \label{fig:exp1_1}
\end{figure}

\begin{figure}[h]
    \centering
    \includegraphics[width=\textwidth]{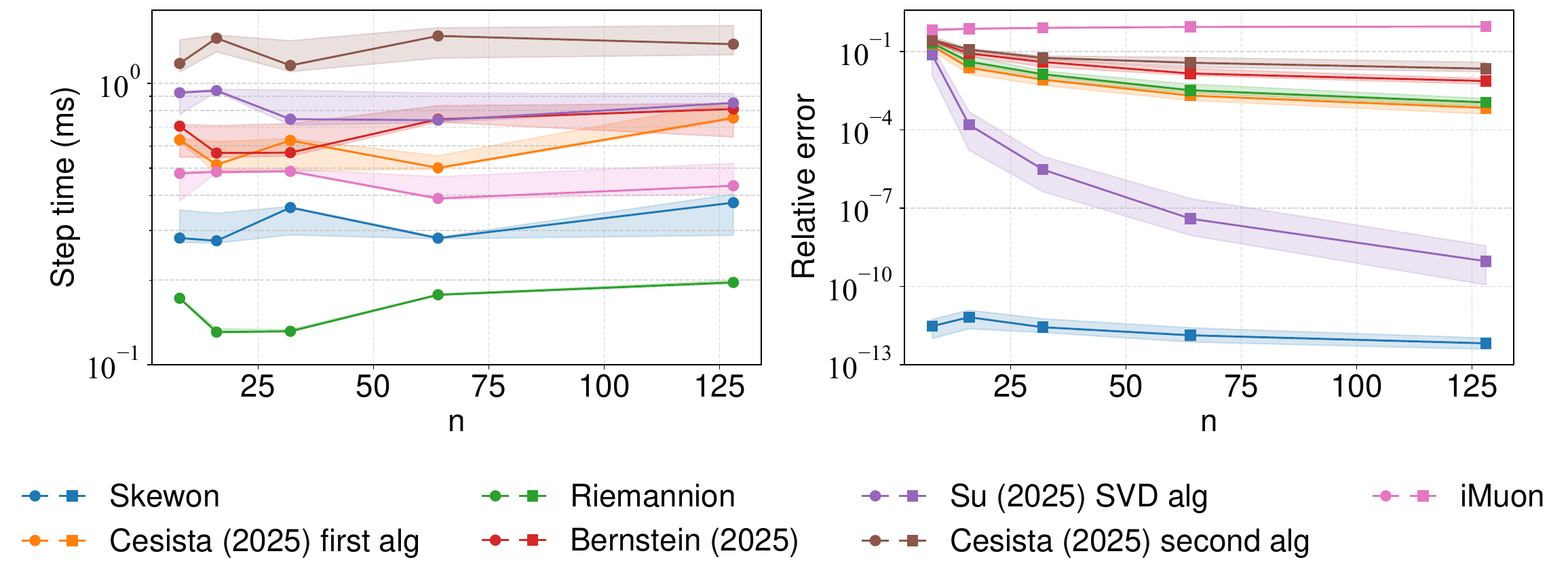}
    \caption{The left graph shows the dependence of the SMP~\eqref{eq:stiefel_muon_problem} solution median time on \(n\). The right graph shows the median relative error of the SMP~\eqref{eq:stiefel_muon_problem} solution at different \(n\). Parameter \(p\) is fixed and equal to \(4\). The range of $0.25$ and $0.75$ quantiles is shaded.}
    \label{fig:exp1_2}
\end{figure}

\begin{figure}[h]
    \centering
    \includegraphics[width=\textwidth]{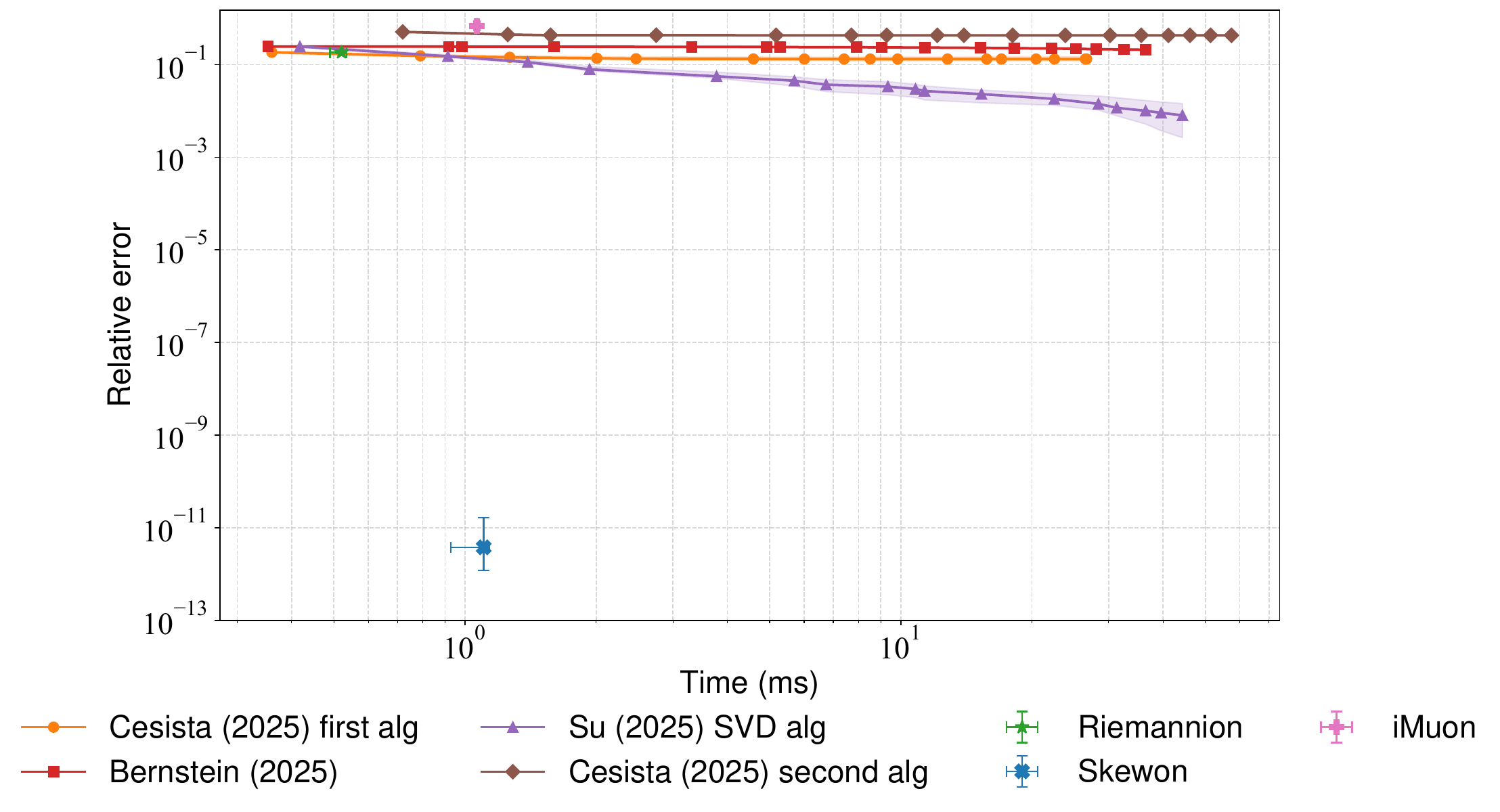}
    \caption{The graph shows the dependence of the relative error on the solution median time for solving the SMP~\eqref{eq:stiefel_muon_problem}. Parameters \(n,p\) are fixed and equal to \(64,32\) respectively. The range of $0.25$ and $0.75$ quantiles is shaded.}
    \label{fig:exp1_3}
\end{figure}

\section{Conclusion}

We presented an exact solution to the spectral linear minimization oracle arising in Muon optimization on the Stiefel manifold. By reformulating the original tangent-space problem as an equivalent optimization problem over skew-symmetric matrices, we derived a closed-form characterization of the optimal solution and proposed the Skewon algorithm for its efficient evaluation. We also established first-order convergence guarantees for the resulting optimization method under standard assumptions in Riemannian optimization.

\bibliography{sn-bibliography}

\begin{appendices}

\section{Additional proofs}\label{sec:proofs}
\subsection{Proof of Proposition~\ref{lemma:boundary_optimum_stifel_muon}}\label{lemma:boundary_optimum_stifel_muon_proof}

Every optimal solution of the SMP~\eqref{eq:stiefel_muon_problem} with nonzero $\mathcal{P}_XM$ is attained on the boundary of the
spectral-norm unit ball, i.e. $\left\|B_{\mathrm{opt}}\right\|_2 = 1$.
\begin{proof}
First we show that for a nonzero \(\mathcal{P}_XM\) every solution of the SMP~\eqref{eq:stiefel_muon_problem} has a strictly negative objective value:
\[
    \langle M, B_{\mathrm{opt}} \rangle < 0.
\]  
This follows, for example, from the fact that one can take $\hat B = -\mathcal{P}_X M /\|\mathcal{P}_X M\|_2$. Since \(\mathcal{P}_X M \in T_X \mathrm{St}(n,p)\) we have \(\hat B \in T_X \mathrm{St}(n,p)\), \(\|\hat B\|_2 = 1\) and 
\[
    \langle M, \hat B\rangle = -\|\mathcal{P}_XM\|^2_F / \|\mathcal{P}_XM\|_2 < 0.
\]
The optimal solution \(B_{\mathrm{opt}}\), on the other hand, cannot have a larger objective value.

Now suppose for contradiction, that a solution \(B_{\mathrm{opt}} \in T_X \mathrm{St}(n,p)\) lies in the interior of the spectral-norm ball, that is,
\[
0 <\|B_{\mathrm{opt}}\|_2 < 1.
\]
Let 
\begin{equation*}
    \tilde{B} = \frac{B_{\mathrm{opt}}}{\|B_{\mathrm{opt}}\|_2}.
\end{equation*}
Since the Stiefel tangent space is linear, $\tilde{B}$ remains feasible for the problem, and by construction, $\|\tilde{B}\|_2 = 1$. However, its objective value is
\begin{equation*}
    \langle M, \tilde{B} \rangle = \frac{1}{\|B_{\mathrm{opt}}\|_2} \langle M, B_{\mathrm{opt}} \rangle < \langle M, B_{\mathrm{opt}} \rangle.
\end{equation*}
This contradicts the assumption that $B_{\mathrm{opt}}$ is the optimal solution. Hence, the solution must satisfy $\|B_{\mathrm{opt}}\|_2 = 1$.
\end{proof}

\subsection{Proof of Theorem~\ref{theorem:convergence_garancy}}\label{theorem:convergence_garancy_proof}
Under Assumption~\ref{ass:1}, for any realization $\{X_t\}_{t\in \mathbb{N}}$ of Algorithm~\ref{alg:1} or Algorithm~\ref{alg:2}, let $\Delta \geq f(X_0) - \inf_{X \in \mathrm{St}(n,p)} f(X)$, and choose the optimization step size $\eta = \sqrt{{2\Delta}/{pLT}}$. Then,
\[
\min_{t\in\{0,\ldots,T-1\}} \left\|\mathrm{grad}f(X_t)\right\|_F \leq \sqrt{\frac{8p\Delta L}{T}}.
\]

\begin{proof}
Let $\mathrm{skew}\left(\nabla f(X_t)X_t^\top\right) = U_t\Sigma_t V_t^\top$ --- any thin SVD. By defining solution of \eqref{eq:skewon_problem_2} in Algorithm~\ref{alg:1} or Algorithm~\ref{alg:2} as 
\[
B_t = -\mathrm{skew}\left(U_tV_t^\top\right)X_t,
\]
we get 
\[
\langle \nabla f(X_t), B_t \rangle = -\| \mathrm{skew}\left( \nabla f(X_t)X_t^\top\right) \|_*.
\]
Note that since $B_t \in T_{X_t}\mathrm{St}(n,p)$, then 
\begin{align*}
\langle \nabla f(X_t), B_t \rangle = \langle \nabla f(X_t), \mathcal{P}_{X_t}B_t \rangle = \langle \mathcal{P}_{X_t}\nabla f(X_t), B_t \rangle =  \langle \mathrm{grad} f(X_t), B_t \rangle.
\end{align*}
Then, using Assumption~\ref{ass:1}:
\begin{align*}
f(X_{t+1}) &\leq f(X_t) + \eta_t \langle \mathrm{grad} f(X_t), B_t \rangle + \frac{\eta^2_tL}{2}\|B_t\|_F^2 \leq \\
&\leq f(X_t) - \eta_t \|\mathrm{skew}\left( \nabla f(X_t)X_t^\top\right)\|_* + \frac{\eta^2_tpL}{2}.
\end{align*}
Telescoping for $t = 0,..., T-1$, we get
\begin{align*}
\sum_{t=0}^{T-1} \eta _t\left\|\mathrm{skew}\left( \nabla f(X_t)X_t^\top\right)\right\|_* \leq f(X_0) - f(X_T) + \frac{pL}{2} \sum_{t=0}^{T-1} \eta_t^2 \leq \Delta + \frac{pL}{2} \sum_{t=0}^{T-1} \eta_t^2,
\end{align*}
where $\Delta$ is exactly finite due to compactness of the Stiefel manifold. For a constant step $\eta_t=\sqrt{\frac{2\Delta}{pLT}}$, we obtain
\[
\min_{t\in\{0,...,T-1\}} \left\|\mathrm{skew}\left( \nabla f(X_t)X_t^\top\right)\right\|_* \leq \sqrt{\frac{2\Delta p L}{T}}.
\]
Note that in fact, for all $X\in \mathrm{St}(n,p)$ the Riemannian gradient in the canonical metric ($I_n-\frac{1}{2}XX^\top$) is given as $\mathrm{grad^c}f(X) = 2\mathrm{skew}\left(\nabla f(X)X^\top\right)X$. We can write $2\|\mathrm{skew}\left( \nabla f(X)X^\top\right)X \|_* \leq 2\|\mathrm{skew}\left(\nabla f(X)X^\top\right)\|_*$. And thus the guarantees of convergence have already been obtained. However, we will reduce to the Riemannian gradient in our metric. 

Consider the following expression for all $X\in \mathrm{St}(n,p)$:
\begin{align*}
\mathrm{skew}( \nabla f(X)X^\top)X &= \frac{1}{2}\left(\nabla f(X)-X\nabla f(X)^\top X\right) = \\ &=  \frac{1}{2}\left(\nabla f(X)-X\nabla f(X)^\top X \pm XX^\top \nabla f(X) \right) = \\ &= \frac{1}{2}\left(I_n+XX^\top\right)\mathrm{grad}f(X).\end{align*}
Thus, using the invertibility of the matrix $I_n + XX^\top$ and Woodbury formula, we have
\begin{align*}
\left\|\mathrm{grad}f(X)\right\|_F &= 2\left\| \left(I_n - \frac{1}{2}XX^\top\right)\mathrm{skew}( \nabla f(X)X^\top)X \right\|_F \leq \\ &\leq2\left\| \left(I_n - \frac{1}{2}XX^\top\right)\right\|_2 \left\|\mathrm{skew}( \nabla f(X)X^\top)X \right\|_F \leq \\ &\leq
2 \left\|\mathrm{skew}( \nabla f(X)X^\top)X \right\|_F \leq \\ &\leq 2 \left\|\mathrm{skew}( \nabla f(X)X^\top)X \right\|_* \leq 2\|\mathrm{skew}\left(\nabla f(X)X^\top\right)\|_*. 
\end{align*}
Therefore
\[
\min_{t\in\{0,...,T-1\}} \left\|\mathrm{grad}f(X_t)\right\|_F \leq \sqrt{\frac{8\Delta p L}{T}}.
\]
\end{proof}

\subsection{SVD polar factor in rank deficient case}
\begin{proposition}\label{prop:skewSVD} 
Let $N \in \mathcal{A}^n$ have rank $r=2k$. Let $N = U\Sigma V^T$ be an SVD. Define 
\[ 
Q = UV^T, \qquad S = \frac{1}{2}(Q-Q^T). 
\] 
Then $S$ has $2k$ singular values equal to $1$, and the remaining $n-2k$ singular values lie in the interval $[0,1]$: 
\[ 
\sigma(S) = \big(\underbrace{1,\ldots,1}_{2k},\,\mu_1,\ldots,\mu_{n-2k}\big), \qquad \mu_j \in [0,1] \ \text{ for all } j. 
\] 
\end{proposition}
\begin{proof}
By the real Schur decomposition for skew-symmetric matrices, there exist an orthogonal
matrix \(W\) and positive numbers
\(\lambda_1,\dots,\lambda_k\) such that
\[
N
=
W\Lambda W^T,
\qquad
\Lambda
=
\operatorname{diag}
(\lambda_1J,\dots,\lambda_kJ,0_{n-2k}),
\]
where
\[
J=
\begin{pmatrix}
0&1\\
-1&0
\end{pmatrix}.
\]

For any nonzero block $\lambda_i J$, its SVD takes the form $\lambda_i J = U_i(\lambda_i I_2)V_i^T$. Because $\lambda_i I_2$ is nonsingular, the matrix $U_i V_i^T$ is uniquely determined. Since $\lambda_i J = J(\lambda_i I_2)$ is a valid polar decomposition, it follows that $U_i V_i^T = J$ independently of the choice of singular vectors.

For the zero block, however, the orthogonal factor corresponding to this block can be an arbitrary orthogonal matrix $Q_0 \in O(n-2k)$.

Consequently, the matrix $Q$ has the form
\[
Q  = W \operatorname{diag}(J,\dots,J,Q_0) W^T.
\]
Since $J^T = -J$, the skew-symmetric part of $Q$ is given by
\[
S = \frac{1}{2}(Q-Q^T) = W \operatorname{diag}\left(J,\dots,J,\frac{1}{2}(Q_0-Q_0^T)\right) W^T.
\]
Because $W$ is orthogonal, the singular values of $S$ are the union of the singular values of its diagonal blocks. Each $2 \times 2$ block $J$ is an orthogonal matrix and therefore contributes two singular values equal to $1$, yielding $2k$ ones in total. 

Let $S_0 = (Q_0-Q_0^T)/2$ be the remaining block. Since $Q_0$ is orthogonal, the spectral norm of $S_0$ is bounded by
\[
\|S_0\|_2 \le \frac{1}{2} \big( \|Q_0\|_2 + \|Q_0^T\|_2 \big) = 1.
\]
Since the spectral norm bounds the maximum singular value, the $n-2k$ singular values of $S_0$, denoted as $\mu_1,\dots,\mu_{n-2k}$, must lie in the interval $[0,1]$. Thus,
\[
\sigma(S) = \big( \underbrace{1,\dots,1}_{2k}, \mu_1,\dots,\mu_{n-2k} \big),
\]
which completes the proof.
\end{proof}

\end{appendices}

\end{document}